\documentclass{article}

\usepackage{amsmath,amssymb,amsthm}
\usepackage{a4wide}

\newtheorem{theorem}{Theorem}

\newtheorem{proposition}[theorem]{Proposition}
\newtheorem{definition}[theorem]{Definition}

\newtheorem{remark}[theorem]{Remark}

\begin{document}

\title{Cauchy's formula on nonempty closed sets
and a new notion of Riemann--Liouville
fractional integral on time scales\thanks{This is 
a preprint of a paper whose final and definite form is published 
by 'Applied Mathematics Letters' (ISSN:~0893-9659). 
Please cite this article as: \emph{D.F.M. Torres, Cauchy's 
formula on nonempty closed sets and a new notion of Riemann--Liouville 
fractional integral on time scales, 
Appl. Math. Lett. 121 (2021), Art.~107407, 6~pp.}, 
{\tt https://doi.org/10.1016/j.aml.2021.107407}.}}

\author{Delfim F. M. Torres\\
\texttt{delfim@ua.pt}}

\date{R\&D Unit CIDMA, Department of Mathematics,\\ 
University of Aveiro, 3810-193 Aveiro, Portugal}

\maketitle


\begin{abstract}
We prove Cauchy's formula for repeated integration on time scales.
The obtained relation gives rise to new notions of fractional 
integration and differentiation on arbitrary nonempty closed sets. 

\medskip


\medskip

\end{abstract}


\section{Introduction}

Fractional calculus (FC), the study of integration and differentiation 
of non-integer order, is an old subject of current interest \cite{MR3888407}.
On the set $\mathbb{T} = \mathbb{R}$ of real numbers, one can argue that
FC, as a theory, has its origins in the 1823 work of Abel \cite{Abel:1823}
on the tautochrone or isochrone problem, that is, the problem of finding 
the curve for which the time taken by an object sliding without friction 
in uniform gravity to its lowest point is independent of its starting point
on the curve \cite{MyID:310}. FC on $\mathbb{T} = \mathbb{R}$ has been extensively studied, 
with many books on the topic, the encyclopedic reference being \cite{MR1347689}.
Similarly, on the discrete setting $\mathbb{T} = \mathbb{Z}$ the pioneer work 
of FC appears in 1957 in connection with Kuttner's study of difference sequences 
\cite{MR0094618:1957}. A few years later, in 1966, the study of fractional quantum 
calculus on the set $\mathbb{T} = q^\mathbb{Z}$, $0 < q  < 1$, 
of quantum numbers was initiated by Al-Salam \cite{MR0197637:1966}.
Two decades later, with the introduction in 1988, 
by Aulbach and Hilger \cite{MR1062633:1988},
of time-scale calculus on any nonempty closed set $\mathbb{T}$,
the question as whether there exists a unified theory of FC 
on an arbitrary time scale $\mathbb{T}$ became a natural one.
The first works, which considered the question of developing 
a FC on a generic time scale, are three 2011 papers by Bastos et al.
\cite{MR2728463:2011,MyID:179:2011,MR2800417:2011}.
The subject is nowadays under strong current development: 
see \cite{MR3891804:2019,MR3993335:2019,MR4000044:2019} 
and references therein.

One of the approaches to define fractional integration on time scales
is related with the generalization of Euler's gamma function $\Gamma$
to an arbitrary time scale $\mathbb{T}$. Such approach is followed
in the quantum setting by Al-Salam \cite{MR0197637:1966} 
and Agarwal \cite{MR0247389:1969}, where the $q$-gamma function $\Gamma_q$
is used to define fractional $q$-integrals. Such path is problematic
on arbitrary time scales, because of the difficulty to define
a suitable gamma function on $\mathbb{T}$ \cite{MR3136530:2013}.
A different approach is proposed in \cite{MyID:328:2016},
where the fractional integral on time scales is introduced
by integrating on the time scale but making use of the classical
Euler's gamma function $\Gamma$. Although such notion is now being used, 
with success, in several contexts and by different authors, see, e.g.,
\cite{MyID:438,MyID:403:2017,MyID:365:2018,MR3534135:2016,MR3874505:2018}, 
here we show that the definition of \cite{MyID:328:2016}
is not the most natural one on time scales.

The paper is organized as follows. In Section~\ref{sec:2},
we motivate the new definition of fractional integral
on time scales, explaining its difference
with respect to the one of \cite{MyID:328:2016}.
Our central result, the Cauchy formula for repeated 
integration on time scales, in then proved in Section~\ref{sec:3}.
We proceed with Section~\ref{sec:4}, giving the fundamental 
definition of any FC: a notion of fractional integration 
-- see Definition~\ref{def:FI} of fractional integral on time scales 
in the sense of Riemann--Liouville. We end with Section~\ref{sec:conc}
of conclusion, commenting on some possible future directions of research
based on the results here presented.


\section{Motivation}
\label{sec:2}

For an introduction to the calculus on time scales, 
we refer the reader to the comprehensible monograph  
\cite{MR1843232}. Brief but sufficient preliminaries
on the calculus on time scales are found in Section~2
of \cite{MyID:328:2016}. Here we motivate the current
investigation on repeated integration on time scales,
showing why the Riemann--Liouville fractional integral
on time scales given by Definition~10 of \cite{MyID:328:2016}
is not the best one. Let $\mathbb{T}$ be a time scale
with forward operator $\sigma$ and $\Delta$ derivative
and integral. The product rule asserts that
\begin{equation}
\label{eq:PR}
(f \cdot g)^\Delta = f^\Delta \cdot g + f^\sigma \cdot g^\Delta.
\end{equation}
For this reason, while the usual continuous calculus 
of $\mathbb{T} = \mathbb{R}$ asserts that $\left(t^2\right)' = 2t$,
on a general time scale $\mathbb{T}$ one gets 
\begin{equation}
\label{eq:der:t2}
\left(t^2\right)^\Delta = (t \cdot t)^\Delta = t + \sigma(t).
\end{equation}
For $\mathbb{T} = \mathbb{R}$ one has $\sigma(t) = t$ and we get
from \eqref{eq:der:t2} the expected equality $\left(t^2\right)' = 2t$;
but for $\mathbb{T} = \mathbb{Z}$, for example, 
one has $\sigma(t) = t + 1$ and \eqref{eq:der:t2} is telling us
that the forward difference of $t^2$ is $\Delta\left(t^2\right) = 2t + 1$.
Relation \eqref{eq:der:t2} also means that while in $\mathbb{T} = \mathbb{R}$
$\int_{0}^{t} 2 s ds = t^2$, on a general time scale $\mathbb{T}$ the corresponding
equality is
$$
\int_{0}^{t} (s + \sigma(s)) ds = t^2.
$$
These simple examples should be enough for the reader to suspect
that the definition of fractional integral on time scales proposed 
in \cite{MyID:328:2016}, that is, 
\begin{equation}
\label{eq:wrong:FI}
\left({_{a}^{\mathbb{T}}I}_{}^{\alpha} f\right)(t)
:= \frac{1}{\Gamma(\alpha)} \int_{a}^{t} (t-s)^{\alpha-1} f(s)\Delta s,
\end{equation}
is not the natural one on time scales. Indeed, here we claim 
that the correct definition should be
$$
\left({_{a}^{\mathbb{T}}I}_{}^{\alpha} f\right)(t)
:= \frac{1}{\Gamma(\alpha)} \int_{a}^{t} (t-\sigma(s))^{\alpha-1} f(s)\Delta s
$$
(note the appearance of $\sigma$). To show that, we generalize 
Cauchy's formula for repeated integration on time scales, proving
that for $n = 1, 2$ one has
\begin{equation}
\label{eq:CFn}
\left({_{a}^{\mathbb{T}}I}_{}^{n} f\right)(t)
= \frac{1}{\Gamma(n)} \int_{a}^{t} (t-\sigma(s))^{n-1} f(s)\Delta s,
\end{equation}
where ${_{a}^{\mathbb{T}}I}_{}^{n} f$ denotes $n$-times integration 
on the time scale $\mathbb{T}$, that is,
\begin{equation}
\label{eq:def:In}
\left({_{a}^{\mathbb{T}}I}_{}^{n} f\right)(t_n)
:= \int_{a}^{t_n} \cdots \int_{a}^{t_1} f(t_0)\Delta t_0 \cdots \Delta t_{n-1},
\quad n = 1,2
\end{equation}
(cf. Theorem~\ref{thm:CFTS} in Section~\ref{sec:3}).


\section{Cauchy's formula for repeated integration}
\label{sec:3}

There is nothing to prove when $n = 1$, because in this case \eqref{eq:CFn} reduces to
\begin{equation*}
\left({_{a}^{\mathbb{T}}I}_{}^{1} f\right)(t) =  \int_{a}^{t} f(s)\Delta s,
\end{equation*}
which is just definition \eqref{eq:def:In} for $n = 1$.
We prove \eqref{eq:CFn} when $n = 2$. 

\begin{proposition}
\label{prop1}
Let $\mathbb{T}$ be a time scale with $a, t, \tau \in \mathbb{T}$,
$t > a$ and $\tau > a$, and $f$ an integrable function on $\mathbb{T}$. 
Then,
\begin{equation}
\label{eq:RI:n=2}
\int_{a}^{t} \int_{a}^{\tau} f(s) \Delta s \Delta\tau
= \int_{a}^{t} (t -\sigma(s)) f(s) \Delta s.
\end{equation}
\end{proposition}

\begin{proof}
Let $g(t)$ be the right hand side of \eqref{eq:RI:n=2}:
\begin{equation}
\label{eq:g:n=2}
g(t) := \int_{a}^{t} (t -\sigma(s)) f(s) \Delta s.
\end{equation}
Observe that $g(t)$ can be written as
\begin{equation}
\label{eq:g:n=2:expand}
g(t) = t \int_{a}^{t} f(s) \Delta s 
- \int_{a}^{t} \sigma(s) f(s) \Delta s.
\end{equation}
Differentiating \eqref{eq:g:n=2:expand}, we get from the product rule
\eqref{eq:PR} and the fundamental theorem of the calculus on time scales that
\begin{equation}
\label{eq:gdelta:n=2}
g^\Delta(t) = \left[\int_{a}^{t} f(s) \Delta s 
+ \sigma(t) f(t)\right] - \sigma(t) f(t) 
= \int_{a}^{t} f(s) \Delta s. 
\end{equation}
Since by definition \eqref{eq:g:n=2} of $g(t)$ 
one has $g(a) = 0$, we know that
\begin{equation}
\label{eq:g:n=2:TFCI}
g(t) = g(t) - g(a) = \int_{a}^{t} g^\Delta(\tau) \Delta \tau.
\end{equation}
Using \eqref{eq:gdelta:n=2} in \eqref{eq:g:n=2:TFCI}, 
we arrive to 
$$
g(t) = \int_{a}^{t} \int_{a}^{\tau} f(s) \Delta s \Delta\tau,
$$
which proves the intended relation.
\end{proof}

We now state Cauchy's formula of repeated integration
on time scales. Proposition~\ref{prop1} is just the particular 
case of Theorem~\ref{thm:CFTS} with $n = 2$.

\begin{theorem}[Cauchy's result on time scales]
\label{thm:CFTS}
Let $n \in \{1,2\}$, $\mathbb{T}$ be a time scale with 
$a, t_1, \ldots, t_n \in \mathbb{T}$,
$t_i > a$, $i = 1, \ldots, n$, and 
$f$ an integrable function on $\mathbb{T}$. 
Then,
\begin{equation}
\label{eq:RI}
\int_{a}^{t_n} \cdots \int_{a}^{t_1} f(t_0) \Delta t_0 \cdots \Delta t_{n-1}
= \frac{1}{(n-1)!}\int_{a}^{t_n} (t_n -\sigma(s))^{n-1} f(s) \Delta s.
\end{equation}
\end{theorem}

It is important to note that \eqref{eq:RI} does not hold for $n > 2$,
for example for $\mathbb{T}=\mathbb{Z}$. 

\begin{remark}
Let $g(t_n)$ be the right hand side of \eqref{eq:RI}:
\begin{equation}
\label{eq:g}
g(t) := \frac{1}{(n-1)!}\int_{a}^{t} (t -\sigma(s))^{n-1} f(s) \Delta s.
\end{equation}
By the binomial theorem, we observe that $g(t)$ can be written as
\begin{equation}
\label{eq:gaftBT}
\begin{split}
g(t) &= \frac{1}{(n-1)!} \int_{a}^{t} 
\sum_{k=0}^{n-1} \left[\binom{n-1}{k} (-1)^k t^{n-1-k} \sigma^k(s)\right] f(s) \Delta s\\
&= \sum_{k=0}^{n-1} \frac{(-1)^k}{k! (n-1-k)!} t^{n-1-k}
\int_{a}^{t} \sigma^k(s) f(s) \Delta s.
\end{split}
\end{equation}
Differentiating \eqref{eq:gaftBT}, we get from the product rule
and the fundamental theorem of the calculus on time scales that
\begin{equation}
\label{eq:gdelta}
g^\Delta(t) = \sum_{k=0}^{n-2} \left[
\frac{(-1)^k \left(t^{n-2-k} + \sigma(t) \left(t^{n-2-k}\right)^\Delta\right)}{k! (n-1-k)!} 
\int_{a}^{t} \sigma^k(s) f(s) \Delta s\right]
+ \sum_{k=0}^{n-1} \frac{(-1)^k}{k! (n-1-k)!} \, \sigma^{n-1}(t) f(t).
\end{equation}
Since 
$$
\sum_{k=0}^{n-1} \frac{(-1)^k}{k! (n-1-k)!} = 0,
$$
it follows from \eqref{eq:gdelta} that
\begin{equation}
\label{eq:gdelta1}
g^\Delta(t) = \sum_{k=0}^{n-2} 
\frac{(-1)^k \left(t^{n-2-k} + \sigma(t) \left(t^{n-2-k}\right)^\Delta\right)}{k! (n-1-k)!} 
\int_{a}^{t} \sigma^k(s) f(s) \Delta s.
\end{equation}
Observe that, in a general time scale $\mathbb{T}$, 
it may happen that \eqref{eq:gdelta1} is not delta-differentiable.
\end{remark}

Theorem~\ref{thm:CFTS} provides the foundation
of a proper fractional calculus on time scales.


\section{A proper Fractional Calculus on Time Scales}
\label{sec:4}

Although there is a myriad of different FC on the literature,
all of them begin by defining a notion of fractional
integral and then proceeding from there \cite{MR1347689}. 
The most common definition of fractional integral 
is the one of Riemann--Liouville:
\begin{equation}
\label{eq:def:IRL}
\left({_{a}I}^{\alpha}f\right)(t)
:= \frac{1}{\Gamma(\alpha)} \int_{a}^{t} (t-s)^{\alpha-1} f(s) ds.
\end{equation}
The corresponding Riemann--Liouville fractional derivative of order $\alpha$ 
is then introduced by computing the $n$th order derivative over the fractional integral 
of order $n - \alpha$, where $n$ is the smallest integer greater than $\alpha$,
that is, $n := \lceil \alpha \rceil $:
\begin{equation}
\label{eq:def:DRL}
{_{a}D}^{\alpha}f = \frac{d^n}{dt^n} \left({_{a}I}^{n-\alpha}f\right).
\end{equation}
Another option for computing fractional derivatives 
was proposed by Caputo in 1967 \cite{Caputo:1967},
and consists in interchanging the order of the operators
$\frac{d^n}{dt^n}$ and ${_{a}I}^{n-\alpha}$
in \eqref{eq:def:DRL}:
\begin{equation}
\label{eq:def:DRL:C}
{_{a}D}_{C}^{\alpha}f = {_{a}I}^{n-\alpha}\left(f^{(n)}\right).
\end{equation}

Having in mind that $\Gamma(n) = (n-1)!$, Theorem~\ref{thm:CFTS} shows
that \eqref{eq:CFn} holds for $n \in \{1, 2\}$. The question is what
should be ${_{a}^{\mathbb{T}}I}_{}^{\alpha} f$ when $\alpha$
is a positive real number? The answer should be clear since \eqref{eq:CFn} makes
sense for any nonnegative $n \in \mathbb{R}$. This provides the proper
notion of fractional integral on time scales in the sense of Riemann--Liouville.

\begin{definition}[Riemann--Liouville fractional integral on time scales]
\label{def:FI}
Suppose $\mathbb{T}$ is a time scale, $[a,b]$ is an interval of $\mathbb{T}$,
and $f$ is an integrable function on $[a,b]$. Let $\alpha > 0$.
Then, the (left) fractional integral of order $\alpha$ of $f$ is defined by
\begin{equation}
\label{eq:correct:FI}
\left({_{a}^{\mathbb{T}}I}^{\alpha}f\right)(t)
:= \frac{1}{\Gamma(\alpha)} \int_{a}^{t} (t-\sigma(s))^{\alpha-1} f(s)\Delta s,
\end{equation}
where $\Gamma$ is the gamma function.
\end{definition}

In the case $\mathbb{T} = \mathbb{R}$, one gets from 
Definition~\ref{def:FI} the standard notion \eqref{eq:def:IRL}.
The generalization \eqref{eq:correct:FI} is not trivial, 
in the sense that \eqref{eq:wrong:FI}, used currently
in the literature, also generalizes \eqref{eq:def:IRL}.
However, only Definition~\ref{def:FI} is coherent
with Cauchy's formula (Theorem~\ref{thm:CFTS}).

Using Definition~\ref{def:FI}, a fractional calculus
on time scales can now be developed. As expected, the first steps
of such fractional calculus are the notions of fractional 
differentiation on time scales in the sense
of Riemann--Liouville,
\begin{equation}
\label{eq:def:DRL:TS}
{_{a}^{\mathbb{T}}D}^{\alpha}f
= \left({_{a}I}^{n-\alpha}f\right)^{\Delta^n},
\end{equation}
and Caputo,
\begin{equation}
\label{eq:def:DRL:C:TS}
{_{a}^{\mathbb{T}}D}_{C}^{\alpha}f
= {_{a}I}^{n-\alpha}\left(f^{\Delta^n}\right),
\end{equation}
where $n := \lceil \alpha \rceil $, which are the natural extensions 
of \eqref{eq:def:DRL} and \eqref{eq:def:DRL:C}, respectively. Such 
fractional calculus on time scales is rich and technical 
and its development will be addressed elsewhere.


\section{Conclusion}
\label{sec:conc}

We have proposed a new definition of fractional integral on time scales, 
using Cauchy's formula for repeated integration on time scales.  
There is much further work to be done based on this new notion. 
Some next natural steps are to consider fractional differential equations
on time scales, eigenvalue problems, fractional dynamic inequalities,
which are examples of very active research areas in both time-scale and fractional
communities. About applications, we claim that our
fractional calculus on time scales has a big potential
in mathematical modeling, for example in epidemiology
and consensus problems.


\section*{Acknowledgement}

This work was supported by FCT and the 
Center for Research and Development in Mathematics 
and Applications (CIDMA), project UIDB/04106/2020.



\end{document}